\documentclass[11pt, a4paper]{article}
\title{From the random geometry of conformally invariant systems to the K\"ahler geometry of universal Teichm\"uller space \\
\textit{\small submitted for Notices of the American Mathematical Society}}
\usepackage[T1]{fontenc}

\usepackage[utf8]{inputenc}
\usepackage{mathtools}
\usepackage{lmodern}
\usepackage{amsthm,amsmath,amsfonts,amssymb}
\usepackage{graphicx}
\usepackage[shortlabels]{enumitem}
\usepackage{authblk,nccmath}
\usepackage{cite,url}
\usepackage[margin=1cm]{caption}

\usepackage{multirow}
\usepackage{array}
\newcolumntype{P}[1]{>{\centering\arraybackslash}p{#1}}
\usepackage{longtable,booktabs}

\usepackage[activate={true,nocompatibility},final,tracking=true,kerning=true,spacing=true,factor=1100,stretch=20,shrink=20]{microtype}

\microtypecontext{spacing=nonfrench}

\usepackage{hyperref}
\hypersetup{
    colorlinks=true, 
    linkcolor=black, 
    urlcolor=black, 
    citecolor=black,
    linktoc=all 
}

\allowdisplaybreaks

\setlist[enumerate]{topsep = 1ex, leftmargin=.6cm, itemsep= -3pt}
\setlist[itemize]{topsep = 1ex, leftmargin=.6cm, itemsep= -3pt}

\let\OLDthebibliography\thebibliography
\renewcommand\thebibliography[1]{
  \OLDthebibliography{#1}
  \setlength{\parskip}{1pt}
  \setlength{\itemsep}{2pt}
}

\allowdisplaybreaks

\newtheorem{thm}{Theorem}
\newtheorem{cor}[thm]{Corollary}

\newtheorem{df}[thm]{Definition}

\theoremstyle{definition}

\usepackage[font=normal]{caption}
\usepackage{floatrow}

\usepackage{mathtools}

\usepackage[dvipsnames]{xcolor}

\global\long\def\ii{\mathfrak{i}}

\newcommand{\abs}[1]{\left\lvert #1 \right \rvert}

\newcommand{\brac}[1]{\left \langle #1 \right \rangle}

\newcommand{\mc}[1]{\mathcal{#1}}
\newcommand{\m}[1]{\mathbb{#1}}

\def\ie{i.e.,\,}

\renewcommand\Re{\operatorname{Re}}
\renewcommand\Im{\operatorname{Im}}

\def\SLE{\operatorname{SLE}}

\def\mob{\mathrm{M\ddot{o}b}}
\def\Diff{\operatorname{Diff}}
\def\QS{\operatorname{QS}}

\def\Vect{\operatorname{Vect}}

\def\a{\alpha}

\def\g{\gamma}

\def\D{\Delta}

\def\t{\theta}

\def\l{\lambda}
\def\k{\kappa}

\def\S{\Sigma}

\def\o{\omega}
\def\O{\Omega}
\def\vare{\varepsilon}

\def\Chat{\hat{\m{C}}}

\def\dd{\mathrm{d}}
\def\vol{\mathrm{vol}}

\newcommand{\ad}[1]{\overline{#1}}

\def\detz{\mathrm{det}_{\zeta}}
\def\P{\m P}

\def\1{\mathbf{1}}

 \newcommand{\splus}{{\scriptstyle +}}

\usepackage[dvipsnames]{xcolor}

 \def \1{\mathbf{1}}

\def\Id{\operatorname{Id}}

\author{Yilin Wang \\
\emph{Institut des Hautes \'Etudes Scientifiques}\\
\protect\url{yilin@ihes.fr}}
\begin{document}

 \maketitle

Two-dimensional random conformal geometry studies conformally invariant random systems in the plane by combining techniques in complex analysis and probability. 
One of the central notions --- the Schramm--Loewner evolution (SLE) --- is a one-parameter family of random fractal curves without self-crossings.  
For some parameter values, SLE curves are scaling limits of interfaces in two-dimensional critical lattice models. Without reference to discrete models, the family of SLE curves can be defined directly in the continuum and is uniquely characterized by two seemingly weak properties: conformal invariance and domain Markov property.
The fact that there is only one free parameter in the definition of SLE is rooted in the universality of Brownian motion, one of the most fundamental concepts in stochastic analysis. 

We will explain how a fundamental functional on the space of Jordan curves arising from SLE --- Loewner energy --- is connected to a seemingly far apart subject: the K\"ahler geometry of universal Teichm\"uller space. 

Universal Teichm\"uller space is an infinite-dimensional complex Banach manifold which contains all Teichm\"uller spaces of hyperbolic surfaces. There are several equivalent ways to describe universal Teichm\"uller space, one of which is by identifying it with a particular family of Jordan curves --- quasicircles.   Universal Teichm\"uller space is endowed with many geometric structures; in particular, it has an essentially unique homogeneous K\"ahler metric. None of those geometric structures seems to be directly related to stochastic processes. However, its K\"ahler potential, defined on the connected component of the circle (by definition, the Weil--Petersson Teichm\"uller space) and called the universal Liouville action, surprisingly coincides with the Loewner energy.

The fundamental reason for this identity, which is hard to believe to be mere coincidence, is still largely mysterious. This article aims to provide a background on Loewner energy, the universal Liouville action, and the intuition behind the proof of the identity. 
We will also discuss the implications of this link and the connections between Loewner energy and other mathematical subjects.  Due to space limitation, we are not able to give complete references. 
Interested readers are invited to consult the survey \cite{yilin_survey} 
for more references on the background materials.


\section*{Conformally invariant simple random curves in the plane}

Scaling limits of critical lattice models provide several important examples of conformally invariant random self-avoiding curves.  Examples include loop-erased random walk, spin cluster interfaces in the critical Ising model, and the level line of a discrete Gaussian free field. We would like to describe what lattice size scaling limits of such curves may look like. 
Physicists predict that conformal invariance emerge for a wide-range of well-chosen statistical mechanics models \cite{BPZ}. 
For instance, the simple random walk in $\delta \m Z^2$ converges in the scaling limit (when $\delta \to 0$) to the planar Brownian motion, whose trajectory is indeed conformally invariant up to reparametrization.
However, since we deal with self-avoiding paths, any scaling limit should still be non-self-crossing, so cannot be a Markov process. 
(The future path of a Markov process only depends on the current location. For instance, the trajectory of a Brownian motion in $\m R^2$ almost surely hits itself infinitely many times on arbitrarily small time interval.)

By considering properties expected from the discrete models as predicted by physics, O.~Schramm formulated a mathematical description of random simple curves with conformal symmetries \cite{Schramm2000}.
It turns out to be convenient to first assume that the random simple curve connects two distinct boundary points (prime ends) $a, b$ in a simply connected domain $D \subsetneq \m C$. We call such a simple curve a \emph{chord} in $(D;a,b)$ and oriented from $a$ to $b$. This setup will allow us to progressively slit open the curve starting from $a$, and so that the remaining part of the curve will be a chord connecting the endpoint of the slit and $b$ in the slit domain. 

One advantage of working in two dimensions is that there is a rich family of conformal maps, given by biholomorphic functions (so we may use single-variable complex analysis). In particular, the Riemann mapping theorem states that there exists a conformal map $\varphi$ sending the simply connected domain $D$ to the upper half-plane $\m H = \{z \in \m C \colon \Im (z) > 0\}$, which maps $a$ to $0 \in \partial \m H = \m R \cup \{\infty\}$ and $b$ to $\infty \in \partial \m H$. Another choice of such conformal map has to be of the form $\lambda \varphi$ for some $\lambda > 0$.

Schramm realized that the random simple chord has to satisfy two properties:  
\begin{itemize}
    \item \emph{Conformal invariance:}  The random chord  $\g$ in $\m H$ connecting $0$ to $\infty$ should have the same law as its image under the map $z \mapsto \lambda z$ for every $\l > 0$. This then allows us to define the random chord in $(D;a,b)$ as the preimage under any $\varphi : D \to \m H$.
    \item \emph{Domain Markov property:} Chords of $(\m H; 0, \infty)$ have a parametrization by capacity (explained below). If we slit the random chord $\g$ in $(\m H; 0,\infty)$ from $\g(0) = 0$ up to $\g(s)$ and map conformally $(\m H \smallsetminus \g [0,s]; \g (s), \infty)$ to $(\m H; 0, \infty)$, then the image of $\g [s,\infty)$ should have the same law as $\g$ and be independent of $\g [0,s]$, for every $s > 0$. 
\end{itemize}

\begin{figure}[ht]
 \centering
 \includegraphics[width=0.8\textwidth]{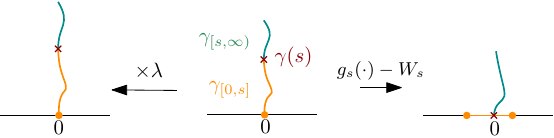}
 \caption{ 
 The left arrow illustrates a scaling map, and the right arrow illustrates the uniformizing conformal map from the slit domain onto the upper half-plane. The law of the random chord should be identical in these three pictures. \label{fig:scaling}}  
 \end{figure}

Indeed, many scaling limits of the interfaces in critical lattice models are proved to satisfy these axioms, e.g., \cite{LSW04LERWUST,Smi:ICM}.
However, we do not need to invoke discrete models if we use these two properties as the axiomatic definition of the random chord of interest in the continuum. Schramm noticed that the Loewner transform \cite{Loewner23} provides a perfect tool to describe it. The \emph{Loewner transform} encodes any given chord in $(\m H; 0, \infty)$ into a continuous real-valued function. More precisely, we say that $\g$ is \emph{parametrized by capacity} if the unique conformal map 
$g_t : \m H \smallsetminus \g [0,t] \to \m H$,  normalized such that the expansion at infinity is given by $g_t (z) = z + o(1)$,  has in fact the next term in the expansion  being $2t/z$, namely, $$g_t (z) = z + \frac{2t}{z} + o \Big( \frac1z \Big). $$
The image of the tip  $\g(t)$ of $\g[0,t]$ under $g_t$ is denoted as $W_t \in \m R$. 
Let $T \in (0,\infty]$ be the total capacity of $\g$. 
As $t$ varies in $[0,T)$, $t \mapsto W_t$ is continuous and starting from $W_0 = 0$. We call $W$ the \emph{Loewner driving function} of $\g$. Moreover, for each $z \in \m H$, $t\mapsto g_t(z)$ satisfies a simple differential equation $\partial_t g_t (z) = 2/(g_t (z) - W_t)$ with the initial condition $g_0 (z) = z$. This allows us to recover $\g$ from $W$.

It is a simple exercise to check that the two properties of the random simple chord above translate into the following properties of the random continuous driving function $W$:
\begin{itemize}
    \item  $W$ has the same law as the function $t \mapsto \l W_{\l^{-2} t}$, for every $\lambda > 0$; 
    \item for every $s > 0$, $t  \mapsto W_{t +s} - W_s$ has the same law as $W$ and is independent of $W|_{[0,s]}$.
\end{itemize}
The first condition implies that $T = \infty$ almost surely.
It turns out the only possible random continuous function satisfying these properties is of the form of $\sqrt \k B$, where $B$ is the standard Brownian motion and $\k \ge 0$. In fact, the second property guarantees $W$ to be a continuous L\'evy process. The classification of continuous L\'evy processes tells us that it is of the form $t\mapsto \sqrt \k B_t +  a t$, for some $a \in \m R$.  One may view the classification as the manifestation of two most fundamental theorems in probability theory: the law of large numbers (which explains the occurrence of the deterministic drift $t \mapsto a t$), central limit theorem (which explains the occurrence of the Gaussian process $\sqrt k B$).
Since Brownian motion satisfies the first property which transforms the drift $at$ into $\l^{-1} at$, it implies that $\lambda^{-1}a = a$ which shows $a = 0$.

\emph{Schramm--Loewner evolution} $\SLE_\k$ in $(\m H; 0, \infty)$ is the random curve whose driving function is $\sqrt \k B$, for some $\k \ge 0$.
 More precisely, when $\k \le 4$, SLE$_\k$ is indeed the random simple chord with driving function $\sqrt \k B$ in the sense described above; for $\k > 4$, the Loewner driving function does not define a growing slit $\g [0,t]$ but a growing compact set $K_t \subset \ad{\m H}$. In this case, SLE$_\k$ is referred to as the curve which carves out progressively the boundary of $K_t$ (when $\k \ge 8$, SLE$_\k$ is a random space-filling curve) \cite{Rohde_Schramm}. 
We will only consider the case where $\k \le 4$ and do not enter into further details to discuss the $\k > 4$ case. The SLE$_\k$ in another domain $(D;a,b)$ is defined as the preimage of $\SLE_\k$ in $(\m H; 0,\infty)$ under any conformal map $\varphi : D \to \m H$ sending respectively $a, b$ to $0,\infty$. 

 From the discussion above, SLE curves are the only possible non-self-crossing curves appearing in conformally invariant systems which satisfy conformal invariance and domain Markov property.
 The study of SLE curves and their variants, as well as the connections to random surfaces, discrete models, and conformal field theory has grown into a thriving field of two-dimensional random conformal geometry. 
 We refer the reader to the textbooks, e.g.,~\cite{WW_St_Flour}, which provide detailed introductions to SLEs and applications. 


\section*{Loewner energy and SLE}
We will slowly move away from the probability world by only looking at a deterministic functional that arises from large deviation principles of SLE. 
Interested readers may consult \cite{yilin_survey} for a more complete survey on this topic. 

Large deviation principles capture the exponential decay of the probability of rare events. A classical example is the large deviation principle of the Brownian motion: Schilder's theorem states that the probability of $\sqrt \k B$ staying close to a given deterministic real-valued function $W$ decays exponentially fast  as $\k \to 0\splus$ (when $W$ is not the zero function):
$$\lim_{\vare \to 0} \lim_{\k \to 0} - \kappa \log \m P (\sqrt \k B \text{ stays $\vare$-close to } W) =  I(W),$$
where $I(\cdot)$ is called
the \emph{large deviation rate function} of $(\sqrt \k B)_{\k \to 0\splus}$  and is given by the Dirichlet energy
$$I(W) : = \frac12 \int_0^\infty \dot W_t^2 \,\dd t, $$
where $\dot W_t : = \dd W_t/\dd t$ if $W$ is absolutely continuous, and $I(W) = \infty$ otherwise. Here, by $\vare$-close we mean being in a ball of radius $\vare$ with respect to a metric on $C^0([0,\infty))$ inducing uniform convergence on compact subsets. 

\begin{figure}[ht]
    \centering
    \includegraphics[width=0.75\textwidth]{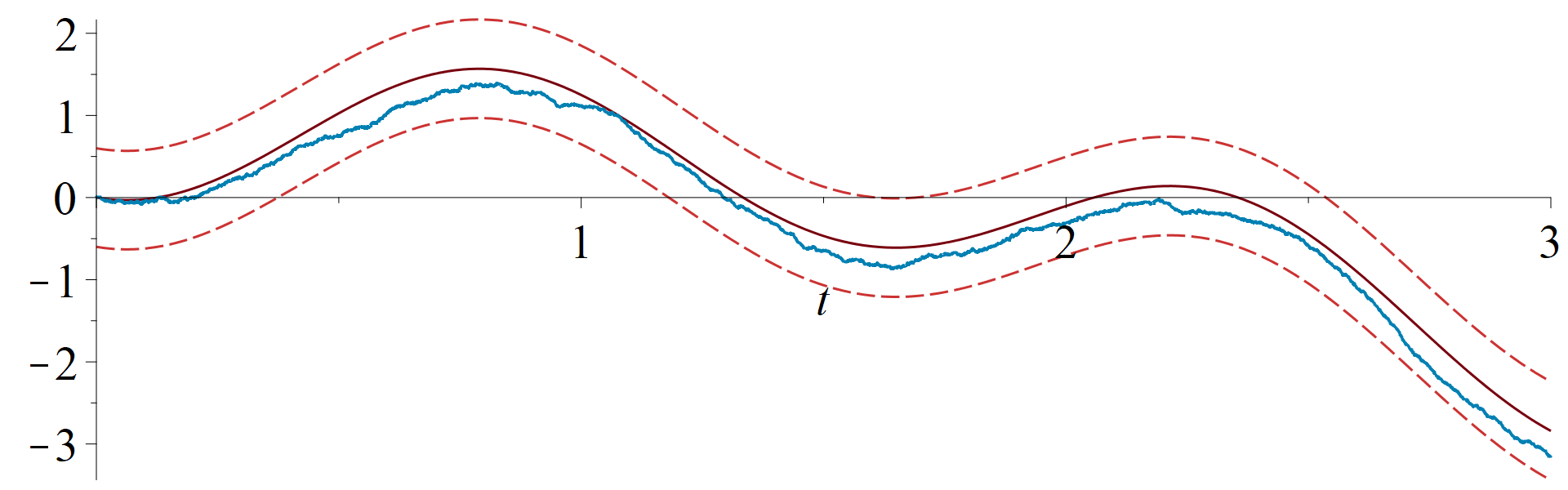}
    \caption{An illustration of the rare event of $\sqrt \k B$ being close to a deterministic function $W$ (whose graph is the solid red line). The blue curve is a simulation of $\sqrt \k B$ with $\kappa = 0.1$ over $10000$ steps conditioned to stay close to $W$.}
    \label{fig:bm}
\end{figure}

Given Schilder's theorem, it is not surprising that using an appropriate metric, e.g., the Hausdorff metric, 
a similar large deviation principle holds for $\SLE_{0+}$: given a chord $\gamma$ in the domain $(D; a,b)$,
\begin{equation}\label{eq:ldp_sle}
\lim_{\vare \to 0} \lim_{\k \to 0} - \k \log \m P (\SLE_\k \text{ in } (D; a,b) \text{ stays $\vare$-close to } \g) = I_{D;a,b}^C(\g). 
\end{equation}
Here 
$I^C_{D;a,b}(\g) : = I(W)$ is called the \emph{chordal Loewner energy} of $\g$, where $W$ is the driving function of $\varphi (\g)$ and $\varphi$ is any conformal map sending $(D; a,b)$ onto $(\m H; 0, \infty)$.  

In a collaboration with Rohde, we generalized the Loewner energy to Jordan curves (simple loops) and realized that it shows more symmetries in this setup \cite{RW}. We will focus on the energy for Jordan curves from now and denote it as $I^L$. The energy for Jordan curve generalizes the chordal energy because of the following property: 
\begin{equation}\label{eq:chord_loop}
    I^L(\g \cup \m R_+) =  I^C_{\m C \smallsetminus \m R_+; 0, \infty} (\g) 
\end{equation}
for every simple chord $\g$ in $(\m C \smallsetminus \m R_+; 0, \infty)$.
The Loewner energy for a Jordan curve is defined as follows.
\begin{df}
 Let $\g: [0,1] \to \Chat = \m C \cup \{\infty\}$ be a continuously parametrized Jordan curve with $\g (0) = \g(1)$. 
For every $\vare>0$, $\g [\vare, 1]$ is a chord connecting $\g(\vare)$ to $\g (1)$ in the simply connected domain $\Chat \smallsetminus \g[0, \vare]$.
    The \emph{Loewner energy} of $\g$ rooted at $\g(0)$ is 
\begin{equation}\label{def:Loewner_energy}
I^L(\g, \g(0)): = \lim_{\vare \to 0} I^C_{\Chat \smallsetminus \g[0, \vare]; \g(\vare), \g(1)} (\g[\vare, 1]).
\end{equation}
\end{df}
In fact, rather surprisingly, the definition does not depend on the choice of the orientation of the curve nor on its root \cite{RW}. 
Therefore, we omit the root from the notation and simply call $I^L(\g)$ the \emph{Loewner energy} of $\g$.
 The relation \eqref{eq:chord_loop} can be seen by putting the root at $\infty$ and orient the curve from $\infty \to 0$ along $\m R_+$.
The independence of the Loewner energy from the parametrization is not obvious, since the chordal energies $I^C$ are defined using the Loewner driving function which depends strongly on the past of the curve. 
This independence suggests that there must be an intrinsic expression of the Loewner energy which does not use any particular parametrization of the Jordan curve.
The answer is given by the identity with the universal Liouville action that will be the subject of the next section. 

 We remark that since the Loewner energy is defined via uniformizing maps (to define the driving function), it is invariant under conformal automorphisms of $\Chat$ (\ie M\"obius transformations $z \mapsto (az +b)/(cz+d)$). 
Moreover, the Loewner energy is zero if and only if $\g$ is a circle on $\Chat$. 
Therefore, the Loewner energy may be viewed as a quantity measuring the roundness of the unparametrized Jordan curve.

The loop energy is, as one may guess, related to the loop version of SLE$_\kappa$. Indeed, the \emph{SLE$_\kappa$ loop measure} $\mu^\kappa$ was constructed in \cite{zhan2020sleloop}, and in this case, the link to the Loewner energy actually holds without letting $\kappa$ go to $0$. A recent work \cite{carfagnini_wang} shows that for a fixed $\kappa \le 4$,
\begin{align}\label{eqn.OM}
\lim_{\varepsilon \rightarrow 0} \frac{\mu^{\kappa}  \left( O_{\varepsilon} (\gamma )\right) }{\mu^{\kappa}  \left( O_{\varepsilon} (S^{1} )\right) } = \exp \left( \frac{c(\kappa)}{24}I^{L}(\gamma) \right),
\end{align}
where $c(\kappa):= (6-\kappa)(3\kappa-8)/2\kappa$ is the \emph{central charge} (terminology coming from the connection to  conformal field theory) of $\SLE_{\kappa}$, $S^1$ is the unit circle, and $O_{\varepsilon}$ is a certain $\varepsilon$-neighborhood of $\gamma$ in the space of loops. In other words, $(c(\kappa)/24) I^L$ is the \emph{Onsager--Machlup action functional}  of the SLE$_\kappa$ loop measure. On may also check that as $\kappa \to 0\splus$, $c(\kappa) \sim -24/\kappa$ and we recover the asymptotics similar to \eqref{eq:ldp_sle}.

\section*{Identity with the Universal Liouville action}

We will focus on the Loewner energy for Jordan curves (and forget about its relation to SLE for a moment).  The following theorem gives an equivalent expression the Loewner energy which does not use the Loewner transform.  Since the Loewner energy is invariant under M\"obius transformations,  without loss of generality, we may assume that the Jordan curve $\g$ does not pass through $\infty$. 

\begin{thm}[See {\cite[Thm.\,1.4]{W2}}]\label{thm:intro_equiv_energy_WP} \label{thm:main}
Let $\O$ \textnormal(resp., $\O^*$\textnormal) denote the component of $\Chat \smallsetminus \g$ which does not contain $\infty$ \textnormal(resp., which contains $\infty$\textnormal) and $f$ \textnormal(resp., $g$\textnormal) be a conformal map from the unit disk $\m D = \{z \in \Chat \colon |z| < 1 \}$ onto $\O$ \textnormal(resp., from $\m D^* = \{z \in \Chat \colon |z| >1 \}$ onto $\O^*$\textnormal). We assume further that $g (\infty) = \infty$.
The Loewner energy of $\gamma$ can be expressed as
       \begin{equation} \label{eq_disk_energy}
   I^L(\gamma) = \frac{1}{\pi} \int_{\m D} \abs{\frac{f''}{f'}}^2 \dd^2 z + \frac{1}{\pi} \int_{\m D^*} \abs{\frac{g''}{g'}}^2 \dd^2 z +4 \log \abs{\frac{f'(0)}{g'(\infty)}} =:  \frac{\mathbf{S} (\g)}{\pi} ,
 \end{equation}
 where $g'(\infty):=\lim_{z\to \infty} g'(z)$ and $\dd^2 z$ is the Euclidean area measure.
\end{thm}

The quantity $\mathbf{S}$ was introduced in \cite{TT06} under the name \emph{universal Liouville action}. Its value does not depend on the choice of $f$ and $g$ as long as $g (\infty) = \infty$.
A Jordan curve for which $\int_{\m D} |f''/f'|^2 \,\dd^2 z$ is finite is called a \emph{Weil--Petersson quasicircle}. It turns out that $\int_{\m D} |f''/f'|^2 \,\dd^2 z$ is finite if and only if $\int_{\m D^*} |g''/g'|^2 \,\dd^2 z$ is finite. Hence, we have:
\begin{cor}
    A Jordan curve has finite Loewner energy if and only if it is a Weil--Petersson quasicircle.
\end{cor}
We note that Weil--Petersson quasicircles are rectifiable. Therefore, finite Loewner energy curves are more regular than SLE$_\k$ curves, which have Hausdorff dimension $(1 + \k /8) \wedge 2$. 
This is not surprising as functions with finite Dirichlet energy are also more regular than a Brownian path which has infinite Dirichlet energy almost surely. 

As mentioned above, the universal Liouville action does not depend on any special point on the curve $\g$ 
and has the advantage of involving only two conformal maps. Whereas to define the Loewner energy through the driving function, one has to study the whole family of uniformizing conformal mappings of the slit domains. 
However, the way of considering the Jordan curve as a progressively growing slit (which closes up on itself) allows us to relate the Loewner energy to SLE curves.

\section*{Weil--Petersson Teichm\"uller space}

The universal Liouville action arises from a very different context --- Teichm\"uller theory.  Teichm\"uller spaces were introduced by Teichm\"uller to parametrize the family of complex structures on a surface using quasiconformal mappings. In particular, the Teichm\"uller space of a genus $g \ge 2$ closed surface is homeomorphic to  $\m R^{6g - 6}$. 
The universal Teichm\"uller space is infinite-dimensional and contains all Teichm\"uller spaces of surfaces of negative Euler characteristics, hence the name \emph{universal}, which can also be represented by the class of quasicircles. 

More precisely, we first identify a Jordan curve $\g$ with a homeomorphism of the unit circle $S^1 = \partial \m D = \partial \m D^*$ as follows. 
By Carath\'eodory's theorem, any conformal map $f : \m D \to \O$ (resp., $g : \m D^* \to \O^*$) extends continuously to a homeomorphism between the closures $\ad {\m D} \to \ad {\O}$ (resp., $\ad {\m D^*} \to \ad {\O^*}$). In particular, $f$ and $g$ restricted to $S^1$ define two homeomorphisms $S^1 \to \g$. The \emph{welding homeomorphism},  defined as the circle homeomorphism $\varphi : = g^{-1} \circ f|_{S^1}$, compares these two homeomorphisms.

The converse operation --- solving the conformal welding problem --- consists of finding a Jordan curve $\g$ and corresponding conformal maps $f$ and $g$, whose welding homeomorphism $g^{-1} \circ f|_{S^1}$ is a given circle homeomorphism $\varphi$. 
We note that if $\g$ is a solution, then $A \circ \g$ is also a solution (by replacing $f$ by $A \circ f$ and $g$ by $A \circ g$), where $A$ is any M\"obius transformation of $\Chat$. 
A solution may not exist, and if one exists, it may not be unique (up to post-composition by M\"obius transformations). 
The complete characterization of circle homeomorphisms for which there is a unique solution to conformal welding is a hard open question: 
Bishop showed that those corresponding Jordan curves form a set which is \emph{not} Borel measurable.  

However, it is classical 
that if the circle homeomorphism is \emph{quasisymmetric}, then the solution to the conformal welding problem exists and is unique up to post-composition by M\"obius transformations.
We use
$$\medmath{\QS(S^1): = \left \{\varphi \in \operatorname{Hom}(S^1) \colon \exists M > 1, \, \forall \t \in \m R, \, \forall t \in (0,\pi), \,  \frac1M \le \abs{\frac{\varphi(e^{\ii (\t + t)}) -\varphi( e^{\ii \t})}{\varphi(e^{\ii \t} )- \varphi(e^{\ii (\t-t)})}} \le M \right\}}$$ 
to denote the group of quasisymmetric circle homeomorphisms.   The corresponding Jordan curves are called \emph{quasicircles}. 

Let $\varphi \in \QS (S^1)$ and consider its solution to the welding problem.
We fix a normalization by assuming that the conformal map $f$ satisfies $f(0) = 0, f'(0) = 1$ and $f''(0) = 0$ and put no condition on $g$ (except that $g (\m D^* ) = \Chat \smallsetminus \overline{f(\m D)}$). 
In other words, we consider the welding homeomorphism to be in the homogeneous space $\mob (S^1) \backslash \QS(S^1)$, where $\mob (S^1)$ is the group of M\"obius transformations preserving $S^1$ (since $g$ can be replaced by $g \circ B$ for any $B \in \mob (S^1)$).
The homogeneous space $T(1) : = \mob (S^1) \backslash \QS(S^1)$ is called the \emph{universal Teichm\"uller space}. 

Universal Teichm\"uller space has an infinite dimensional complex Banach manifold structure such that 
the group $\QS(S^1)$ acts on the right on $T(1)$ holomorphically. One wonders whether it can be further equipped with a \emph{K\"ahler metric}, namely, a symplectic form $\o (\cdot, \cdot)$ and Riemannian metric $\brac{\cdot, \cdot}$ that are invariant under the right action and compatible with the complex structure (encoded in an operator $J$ on the tangent bundle such that $J^2 = -\Id$ which plays the role of ``multiplication by $\ii$''). Being compatible means $\o (\cdot, J (\cdot)) = \brac{\cdot, \cdot}$.
This question was motivated by string theory, e.g., \cite{BowickRajeev1987string}, who consider only the \emph{smooth part} $\mob (S^1) \backslash \Diff^\infty (S^1)$ of the universal Teichm\"uller space without having to  worry about any convergence issue on this infinite dimensional manifold. 
It turns out there is a \emph{unique} K\"ahler metric up to a scaling factor.


Let us explain briefly  how to derive this K\"ahler metric (also ignoring the convergence question).
Concretely, the tangent space at $[\Id] \in T(1)$ consists of real vector fields Vect$(S^1)$ on $S^1$ with Fourier expansion: 
$$v = \sum_{n \neq \pm 1, 0} v_n e^{\ii n\t} \frac{ \partial}{\partial \t}  = \sum_{n \neq \pm 1, 0} v_n e_n \quad \text{ satisfying } \ad v_n = v_{-n} \in \m C, $$
where $e_n :=  e^{\ii n\t}  \partial/ \partial \t $.
We have omitted the Fourier modes for $n \in \{\pm 1, 0\}$ as they generate the Lie algebra of $\mob (S^1)$.

The {\em almost complex structure} $J: \text{Vect}(S^1) \to \text{Vect}(S^1)$ (such that $J^2 = -\Id$) induced from the complex structure is given by the Hilbert transform:
$$J v =  \ii \sum_{n = 2}^{\infty} v_n e_n - \ii \sum_{n = -\infty}^{-2} v_n e_n.$$
 
The $\m C$-basis $\{e_n\}_{n \neq \pm 1, 0}$ of the space of complexified vector fields  
has the Lie bracket
$$[e_m, e_n] =  \ii (n - m) e_{n+m}. $$

\begin{thm}[See {\cite{BowickRajeev1987string}}]\label{thm:unique_WP}
  Up to a scaling factor, there is a unique homogeneous K\"ahler metric on $\mob (S^1) \backslash \Diff^\infty (S^1)$. The symplectic form  
  is given by  
  $$\o (u, v) = -\o (v, u) =   - \a \Im  \left (\sum_{n = 2}^{\infty}  (n^3 - n) u_n \ad {v_n} \right), \qquad \forall u,v \in  \text{Vect}(S^1)$$
for some $\a \in \m R_+$.
The symmetric $2$-tensor $\brac{\cdot, \cdot}$ compatible with $\o$ and $J$
\begin{align*}
\brac{u, v} :=  \o (u, J v) =   \a\Re \left( \sum_{n = 2}^{\infty}  (n^3 - n) u_n \ad{v_n} \right)
\end{align*}
is positive and definite. This metric is called the \emph{Weil--Petersson metric}. 
\end{thm}

\begin{proof}
Assume that $\o$ is a homogeneous symplectic form. Since $\o$ is closed, we have for all $m,n,p \in \m Z$,
\begin{align*}
0= \dd \o  (e_m, e_n, e_p) & =   e_m (\o (e_n, e_p)) + e_n (\o (e_p, e_m))+ e_p (\o (e_m, e_n))\\
&\quad - \o ([e_m, e_n], e_p) - \o ([e_n, e_p], e_m) - \o ([e_p, e_m], e_n).    
\end{align*}
By homogeneity, the first three terms on the right-hand side vanish and we have 
\begin{equation}\label{eq:jacobi}
    \o ([e_m, e_n], e_p) + \o ([e_n, e_p], e_m) + \o ([e_p, e_m], e_n))= 0.
\end{equation}
Moreover, $\o$ has kernel spanned by $e_{-1}, e_{0}$ and $e_1$. 
From these constraints we can determine $\o$ as follows.

By taking $p = 0$, \eqref{eq:jacobi} gives that 
$(n+m) \o(e_m,e_n) = 0$.
 Therefore $\o(e_m, e_n) = 0$ when $m \neq -n $. 
 We let 
 $a_m : = \o (e_m, e_{-m}).$
 Take $p = -m -1$, $n = 1$, \eqref{eq:jacobi} gives
 $$ (1- m)a_{m+1} + (m+2) a_m = 0$$
which implies there exists $\a \in \m C$ such that $a_m = \ii \a (m^3 - m)/2$ for all $m \ge 2$ (hence, for all $m \in \m Z$).
 
 When $u, v \in \Vect(S^1)$,  we have $u_{-m} = \ad{u_m}$ and $v_{-m} = \ad{v_m}$. Therefore,
\begin{align*}
\o(u,v) = &  \frac{\ii \a}{2}  \sum_{ m\neq {\pm 1 , 0}} (m^3 - m) u_m v_{-m}
  =  - \a \Im  \left (\sum_{m = 2}^{\infty}  (m^3 - m) u_m \ad {v_m} \right)
\end{align*}
and the compatible symmetric tensor
\begin{align*}
\brac{u, v} : =  \o (u, J v) 
=  \a\Re \left( \sum_{m = 2}^{\infty}  (m^3 - m) u_m \ad{v_m} \right).
\end{align*}
 We obtain $\a > 0$ from the assumption that
$\brac{\cdot, \cdot}$
 is positive definite.
\end{proof}

Takhtajan and Teo \cite{TT06} extended the infinite-dimensional K\"ahler manifold structure and the Weil--Petersson metric to $T(1)$. 
In fact, the subspace of $u \in \text{Vect}(S^1)$ such that $\brac{u,u} <\infty$ coincides with the  $H^{3/2}$ Sobolev space of vector fields (which is strictly smaller than the tangent space of $T(1)$ which is given by the space of Zygmund vector fields). 
\begin{thm}[See {\cite{TT06}}]
The image of $H^{3/2}$ space of vector fields under the right action by $\QS(S^1)$ on $T(1)$ defines a tangent subbundle.
    The connected component $T_0(1)$ of the associated integral manifold containing $[\Id] \in T(1)$ --- called \emph{the Weil--Petersson Teichm\"uller space} --- is a complete, infinite-dimensional K\"ahler--Einstein manifold with negative curvatures. Moreover, $[\varphi] \in T_0(1)$ if and only if $\varphi = g^{-1} \circ f|_{S^1}$ where $\int_{\m D} |f''/f'|^2 \,\dd^2 z <\infty$. 
\end{thm}
So, $T_0(1)$  is the completion of $\mob (S^1)\backslash \Diff^\infty(S^1)$. Moreover, a Jordan curve is associated with an element in $T_0(1)$ via conformal welding if and only if it is a Weil--Petersson quasicircle. 
There are many equivalent characterizations of the elements in $T_0(1)$, see, e.g., \cite{Bishop_WP} for an extensive summary including several geometric characterizations. 
One special feature of K\"ahler metrics is that they always admit locally defined K\"ahler potential. Here, the Weil--Petersson metric has a globally defined potential given by the universal Liouville action:
 
\begin{thm}[See {\cite[Cor.\,II.4.2]{TT06}}]
    The universal Liouville action ${\mathbf S} 
    : T_0(1) \to \m R_+$ is a K\"ahler potential for the Weil--Petersson metric. That is, up to a positive scaling factor,
    $$\ii \partial \bar\partial {\mathbf S} = \o, $$
    where $\o$ is the symplectic form in Theorem~\ref{thm:unique_WP} and $\partial$ and $\bar \partial$ are the $(1,0)$ and $(0,1)$ components of de Rham differential $\dd$ on the complex manifold $T_0(1)$.
\end{thm}
This shows all information about the unique K\"ahler metric on $T_0(1)$ is recorded in the universal Liouville action, hence, the Loewner energy.



\section*{How did we come up with the identity?}
We now explain how we could come up with the identity between the Loewner energy  and the universal Liouville action in Theorem~\ref{thm:main}  using ideas from random conformal geometry. 
 The  first step to guess the following identity 
 for a Jordan curve passing through $\infty$. 
 
\begin{thm}[See {\cite[Thm.\,1.1]{W2}}]\label{thm:infinite_curve_LE}
    If $\gamma$ is a Jordan curve passing through $\infty$, then
\begin{align*}
  I^L(\gamma) & =  \frac{1}{\pi} \int_{\m H} |\nabla \log \abs{f'}|^2 \dd^2 z +  \frac{1}{\pi} \int_{\m H^*} |\nabla \log \abs{g'}|^2 \dd^2 z  \\
   & = \frac{1}{\pi} \int_{\m H} \abs{\frac{f''}{f'}}^2 \dd^2 z +  \frac{1}{\pi} \int_{\m H^*} \abs{\frac{g''}{g'}}^2  \dd^2 z
\end{align*}
    where $f$ and $g$ map conformally the upper half-plane $\m H$ and the lower half-plane $\m H^*$ onto $H$ and $H^*$, the two components of $\mathbb{C} \smallsetminus \gamma$ respectively, while fixing $\infty$. 
\end{thm}

This theorem can be viewed as a deterministic finite energy analog of the quantum zipper coupling between SLE and Gaussian free field (GFF) \cite{MoT,Dub_couplings} that we now explain. We do not make rigorous statements and only argue heuristically here. 

\emph{Gaussian free field} on a two-dimensional domain is a random real-valued Gaussian Schwartz distribution and the analog of Brownian motion by replacing the time interval by the domain. Similar to Schilder's theorem, the large deviation rate function for $(\sqrt k\Phi)_{\k \to 0\splus}$, where $\Phi$ is a free boundary GFF on a two-dimensional domain $D$, is the associated Dirichlet energy:
  $$\mc D_{D} (\varphi) : = \frac{1}{4\pi} \int_D |\nabla \varphi|^2 \dd^2 z \in [0,\infty], \qquad \forall \varphi \in W^{1,2}_{loc}.$$
  If $\mc D_D (\varphi) <\infty$, we say that $\varphi \in \mc E(D)$. 
A $\sqrt \k$-\emph{quantum surface} is a two-dimensional domain $D$ equipped with a
  Liouville quantum gravity  measure (denoted as $\sqrt \k$-LQG), defined using a regularization of $e^{\sqrt \k \Phi}\dd^2 z$, where $\sqrt \k \in (0,2)$. 

We use the following dictionary illustrating the analogy between the concepts in random conformal geometry (left column) and their finite energy counterparts (right column). 
For a more complete dictionary and background, we refer the readers to \cite{VW1,yilin_survey} and the references therein.

\renewcommand{\arraystretch}{1.06}

\begin{center} 
\begin{longtable}{ l  l  }
 \toprule
    \multicolumn{1}{c}{\bf  SLE/GFF with $\k \to 0+$}  &   \multicolumn{1}{c}{\bf  Finite energy} \\ \hline 
     $\SLE_\kappa$ loop    & Jordan curve $\gamma$ with $I^L(\g) <\infty$ \\
     & i.e.,  a Weil--Petersson quasicircle\\
      \hline
        Free boundary GFF $\sqrt \k \Phi$ on $\mathbb{H}$ (on $\m C$) &  $2u \in \mathcal{E}(\m H)$  ($2\varphi \in \mathcal{E}(\m C)$)   \\ \hline
              $\sqrt \k$-LQG on quantum plane $\approx e^{\sqrt \k \Phi} \dd^2 z$ & measure on $\m C$: $e^{2 \varphi} \,\dd^2 z, \, \varphi \in \mathcal{E}(\m C)$ \\ \hline
  $\sqrt \k$-LQG on quantum half-plane on $\mathbb{H}$ & measure on $\m H$: $e^{2 u} \,\dd^2 z ,  \, u \in \mathcal{E}(\m H)$
  \\ 
    \hline
 Quantum zipper coupling: & A Weil--Petersson quasicircle $\gamma$ cuts $\m C$ \\
     `` $\SLE_\kappa$ cuts an independent quantum & with measure $e^{2\varphi} \dd^2 z$, where $\varphi \in \mathcal{E}(\mathbb{C})$,   \\
       plane $e^{\sqrt \k \Phi} \dd^2 z$ into two independent & into half-planes with measure $e^{2u} \dd^2 z$\\
    quantum half-planes $e^{\sqrt \k \Phi_1}, e^{\sqrt \k \Phi_2}$ ''    &  and $e^{2v}\dd^2 z$  with $u \in \mc E(\m H), v \in \mc E(\m H^*)$, \\
    & and $I^L(\gamma) + \mc{D}_{\m{C}}(2\varphi) = \mc{D}_{\m{H}}(2u) + \mc{D}_{\m{H}^*}(2v)$\\
  \bottomrule 
\end{longtable}
\end{center}

\vspace{-30pt}

 In the last line, two domains $D$ and $D'$ equipped with a measure are considered equivalent if there exist a conformal map $D \to D'$ such that the measure on $D'$ equals the pushforward of the measure on $D$. In particular, if a Jordan curve $\g$ cuts $\m C$ into two domains $H$ and $H^*$ as above and $f$ and $g$ are the conformal maps in Theorem~\ref{thm:infinite_curve_LE},  
then $e^{2\varphi} \dd^2 z$ on $H$ and on $H^*$ are equivalent, respectively, to $e^{2 u} \dd^2 z$ on $\m H$ and  $e^{2 v} \dd^2 z$ on $\m H^*$ where
$$   u =  \varphi \circ f + \log \abs{f'}, \quad v =  \varphi \circ g + \log \abs{g'}.$$
 
The identity  $I^L(\gamma) + \mc{D}_{\m{C}}(2\varphi) = \mc{D}_{\m{H}}(2u) + \mc{D}_{\m{H}^*}(2v)$ is more general than Theorem~\ref{thm:infinite_curve_LE} (which corresponds to the case where $\varphi \equiv 0$) 
and we argue heuristically as follows.  From the quantum zipper coupling, one expects that using an appropriate choice of topology and for small $\k$,
 \begin{align}\label{eq:heuristic_cutting}
 \begin{split}
 `` & \P ( \SLE_\k  \text{ loop stays close to } \gamma,\, \sqrt{\k}\Phi \text{ stays close to } 2 \varphi) \\
 = &\;\P(\sqrt \k \Phi_1 \text{ stays close to } 2 u, \, \sqrt{\k} \Phi_2 \text{ stays close to } 2 v)".    
 \end{split}
 \end{align}
 
  We obtain from the independence between SLE and $\Phi$, the large deviation principle of SLE \eqref{eq:ldp_sle}, and the large deviation principle for GFF 
 \begin{align*}
   `` &  \lim_{\k \to 0} - \k \log  \P(\SLE_\k  \text{ stays close to } \gamma,\, \sqrt{\k} \Phi \text{ stays close to } 2 \varphi) \\ 
     =  & \; \lim_{\k \to 0} - \k  \log  \P(\SLE_\k \text{ stays close to } \gamma) +\lim_{\k \to 0}- \k  \log  \P(\sqrt{\k} \Phi \text{ stays close to } 2 \varphi) \\
     \approx & \; I^L (\gamma) + \mc D_{\m C} (2\varphi)".
 \end{align*}
 
 On the other hand the independence between $\Phi_1$ and $\Phi_2$ gives
 \begin{align*}
`` &  \lim_{\k \to 0} -\k \log  \P(\sqrt{\k} \Phi_1  \text{ stays close to } 2 u,\, \sqrt{\k} \Phi_2 \text{ stays close to }  2 v)\\
\approx   &  \; \mc D_{\m H} (2u) + \mc D_{\m H^*} (2v)".
 \end{align*}
 We obtain the identity  $I^L(\gamma) + \mc{D}_{\m{C}}(2\varphi) = \mc{D}_{\m{H}}(2u) + \mc{D}_{\m{H}^*}(2v)$ from \eqref{eq:heuristic_cutting}.  The formula in Theorem~\ref{thm:infinite_curve_LE} follows by taking $\varphi \equiv 0$.
 
 One technical difficulty to make this argument rigorous lies in choosing the right topologies so that these three equations in quotation marks hold.
 However, once we have guessed the identity in Theorem~\ref{thm:infinite_curve_LE}, the actual proof is purely analytic and straightforward (without mentioning SLE or GFF). We will give the outline of the proof in the next section. 
For interested readers, we remark that starting from Theorem~\ref{thm:main}, more identities around the Loewner energy expanding the dictionary above between random conformal geometry and finite energy objects, are explored in \cite{VW1} and its follow-up papers. The proofs there are also entirely analytic and such a dictionary has provided inspiration for proving theorems on both sides.


\section*{Outline of the proof of Theorem~\ref{thm:main}}

To prove Theorem~\ref{thm:main} (where the Jordan curve does not pass through $\infty$), 
we prove first
Theorem~\ref{thm:infinite_curve_LE} without invoking SLE or GFF as follows.  
\begin{itemize}
    \item We prove first the identity when the curve is of the form of $\m R_+ \cup \eta$, where $\eta$ is a chord in $(\m C \smallsetminus \m R_+; 0,\infty)$ with driving function $W : \m R_+ \to \m R$. More specifically, we treat the following cases:
    \begin{itemize}
        \item when $W_t= a t$, for $t \in  [0,T]$ and $W_t = a T$ for  $t \ge T$ (the computation is technical but straightforward in this case);
        \item when $W$ is piecewise linear by concatenating linear functions;
        \item when $W$ satisfies $I(W) < \infty$ by approximating $W$ by piecewise linear functions.
    \end{itemize}
    \item We deduce the identity for curves of the form $[M, \infty] \cup \eta$  where $\eta$ is a chord in $(\m C \smallsetminus [M, \infty); M,\infty)$. Then we let $M \to \infty$. This proves Theorem~\ref{thm:infinite_curve_LE}.
\end{itemize}

The next step aims at giving a more symmetric description of the Loewner energy by viewing the Jordan curve in the sphere $S^2$, so that the point $\infty$ plays no special role.  For this, we equip $S^2$ with a Riemannian metric $g = e^{2 \varphi} g_0$, conformally equivalent to the round metric $g_0$ (the metric induced from $S^2 \subset \m R^3$). 
Let $\g \subset S^2$ be a \emph{smooth} Jordan curve dividing $S^2$ into  two components $D_1$ and $D_2$. Denote by $\D_{D_i,g}$ the Laplace--Beltrami operator with Dirichlet boundary condition on $(D_i, g)$.
We introduce the functional $\mc H (\cdot, g)$ on the space of smooth Jordan curves:
 \begin{align} \label{eq:def_H}
    \begin{split}
  \mc H (\g, g) : = & \log \detz'(-\D_{S^2, g}) -  \log \vol_g(S^2)  - \log \detz (-\D_{D_1, g}) -  \log \detz (-\D_{D_2,g}),      
    \end{split}
 \end{align}
 where $\detz$ denotes the zeta-regularized determinant.

\begin{thm}[See {\cite[Thm.\,7.3]{W2}}] \label{thm_energy_determinant}
We have the following results:
\begin{enumerate}[(i)]
\item The functional $\mc H$ is conformally invariant, \ie $\mc H (\cdot, g) = \mc H (\cdot, g_0)$;
\item \label{item_energy_determinant}  Let $\g$ be a smooth Jordan curve on $S^2$. 
We have the identity
\begin{align}\label{eq:Loewner_determinant}
I^L(\g)& = 12 \mc H(\g, g) - 12 \mc H(\mc C, g)  = 12 \log \frac{\detz(-\D_{\m D_1, g})\detz(-\D_{\m D_2, g})}{\detz(-\D_{D_1, g}) \detz(-\D_{D_2, g})},
\end{align}
where $\mc C$ is any circle and $\m D_1$ and $\m D_2$ are the two components of the complement of  $\mc C$. 
\end{enumerate}
\end{thm}

Let us make some remarks: 
\begin{itemize}
\item  Since the Loewner energy is nonnegative, \ref{item_energy_determinant} implies that circles minimize $\mc H(\cdot, g)$
 among all smooth Jordan curves. 
 \item We assumed the curve $\g$ to be smooth so that $\detz (-\D_{D_i, g})$ is well-defined. This assumption can possibly be weakened. 
 \item The Polyakov--Alvarez formula gives the a formula for the ratio between $\detz(-\D_{\m D_i, g})$ and $\detz(-\D_{D_i, g})$  which involves conformal maps from $\m D_i$ to $D_i$. See, e.g., \cite{W2} for the explicit formula. We deduce the result by comparing \eqref{eq:Loewner_determinant} to the expression in Theorem~\ref{thm:infinite_curve_LE}.
\end{itemize}

Finally, for a smooth Jordan curve which does not pass through $\infty$, we use Theorem~\ref{thm_energy_determinant} and the Polyakov--Alvarez formula again to deduce the identity in Theorem~\ref{thm:main}. The identity for an arbitrary bounded Jordan curve follows from an approximation argument by smooth Jordan curves.


\section*{Further discussions}
We mention that there are other identities between the Loewner energy and mathematical objects that are seemingly unrelated. We now give two examples.  Theorem~\ref{thm:main} suggests there should be a link between SLE and these objects. However, the precise link is not known.

\textbf{Grunsky operator and Coulomb gas.} With each Jordan curve $\g$ is associated a \emph{Grunsky operator} $G_\g$, defined explicitly using the coefficients of a conformal map $g$ from $\m D^*$ to $\O^*$. 
It is a classical object in geometric function theory and  many properties of $\g$ are equivalent to the operator-theoretic properties of  $G_\g$. 
It was shown in \cite{TT06} that $\g$ is Weil--Petersson if and only if $G_\g$ is Hilbert--Schmidt which is equivalent to $\det (\Id - G_\g^* G_\g)$ being well-defined.  More surprisingly, Theorem~\ref{thm:main} and \cite{TT06} together imply that
\begin{equation*}
    I^L(\g) = -12 \log \det (\Id - G_\g^* G_\g).
\end{equation*}
Using this identity, \cite{johansson2023coulomb} showed that the Loewner energy appears in the constant term of large $n$ asymptotics of the free energy of an $n$-particle Coulomb gas confined in $\O$.

\textbf{Renormalized volume.}
We saw that the Loewner energy is invariant under M\"obius transformations of $\Chat$. Since M\"obius transformations extend to isometries of the hyperbolic $3$-space $\m H^3$ (whose boundary at infinity is identified with the Riemann sphere $\Chat$), it is natural to wonder if there is a geometric quantity in $\m H^3$ that is equal to the Loewner energy  (question raised in the pioneering work \cite{Bishop_WP} which also provided qualitative characterizations of Weil--Petersson quasicircles in terms of minimal surfaces in $\m H^3$). 
The answer is positive. It is shown recently in \cite{epstein_yilin} that the Loewner energy of $\g$ equals the renormalized volume of the $3$-manifold $N_\g$ in $\m H^3$ bounded by the two Epstein surfaces $\S$ and $\S^*$ (these are surfaces in $\m H^3$ canonically determined and bounded by $\g$):
\begin{equation*}
    I^L(\g) = \frac{4}{\pi} \left(\vol (N_\g) - \frac12 \int_{\S \cup \S^*} H \dd A \right),
\end{equation*}
where $H$ is the mean curvature on $\S$ and $\S^*$ and $\dd A$ is the area form induced from (the hyperbolic metric in) $\m H^3$. The result uses the fact that $I^L$ is a K\"ahler potential for the Weil--Petersson metric.

\section*{Acknowledgment}

 I am indebted to Wendelin Werner for numerous discussions and to Yuliang Shen for the introduction to Weil--Petersson Teichm\"uller space. I also thank the referees for many comments which improved the presentation of the article.

\bibliographystyle{abbrv}
\bibliography{ref}
\end{document}